\documentclass{amsart}
\usepackage[T1]{fontenc}
\usepackage{amsmath,amssymb,amsthm,amsfonts,bm,changepage,cite,float,mathtools}
\usepackage{caption}
\usepackage{outlines}
\usepackage{diagbox}
\usepackage{hyperref}
\usepackage[none]{hyphenat}
\usepackage{orcidlink}

\numberwithin{equation}{section}
\newtheorem{theorem}{Theorem}[section]
\newtheorem{lemma}[theorem]{Lemma}
\newtheorem{corollary}[theorem]{Corollary}

\theoremstyle{definition}
\newtheorem{definition}[theorem]{Defintion}

\DeclareRobustCommand{\&}{%
  \ifdim\fontdimen1\font>0pt
    \textsl{\symbol{`\&}}%
  \else
    \symbol{`\&}%
  \fi
}

\DeclarePairedDelimiter{\ceil}{\lceil}{\rceil}
\DeclarePairedDelimiter{\floor}{\lfloor}{\rfloor}

\begin{document}

\begin{center}
{\Large \bf ON THE EXISTENCE OF HAMILTONIAN CYCLES IN HYPERCUBES}
\vspace{8mm}

{\large {\bf Gabriele Di Pietro $^{\orcidlink{0009-0000-9276-1269}}$} and {\bf Marco Rip\`a$^{\orcidlink{0000-0002-6036-5541}}$}}
\end{center}
\vspace{14mm}

\begin{adjustwidth}{1.6 cm}{1.6 cm}
{\bf Abstract.} Building on the results of our previous work on Euclidean leaper tours, considering all integers $k>1$ and $h>0$, we study the existence of Hamiltonian cycles in the vertex set $C(2,k):=\{0,1\}^k$ of the $k$-dimensional hypercube when the Euclidean distance between consecutive vertices is fixed. Since the distance between two vertices of $C(2,k)$ is $\sqrt{h}$ for some integer $h$, the problem amounts to determining for which integers $k$ and $h$ there exists a Hamiltonian cycle whose associated Euclidean distance is $\sqrt{h}$. In this paper, we prove that such cycles exist if and only if $h$ is odd and $1 \leq h \leq k-1$. As a result, for all integers $a \geq 0$, $b \geq a$ with $b>0$, we provide a necessary and sufficient condition for the existence of closed Euclidean $(a,b)$-leaper tours on $2 \times 2 \times \cdots \times 2$ chessboards, where the associated distance equals $\sqrt{a^2+b^2}$.

\vspace{6mm}
\end{adjustwidth}

\section{Introduction} \label{sec:1}
In recent years, the classical Knight's Tour Problem (KTP) \cite{3, 7} has been revised and extended to a wide set of similar pieces belonging to fairy chess \cite{1, 4}. This easily follows by defining the knight as the piece that performs jumps of fixed Euclidean length $\sqrt{1^2+2^2}$.

For clarity, we recall that the expression \emph{Euclidean tour}, first introduced in~\cite{5} 
and later used in~\cite{8}, denotes any Hamiltonian path whose consecutive vertices lie at the same Euclidean distance. A closed Euclidean tour that returns to its starting vertex after $2^k$ moves is, therefore, called a \emph{Hamiltonian cycle}.

Now, let $\mathbb{N} \coloneqq \{1,2,3,\ldots\}$ and $\mathbb{N}_0 \coloneqq \{0,1,2,\ldots\}$. Consequently, given $a \in \mathbb{N}_0$ and $b \in \mathbb{N}$ such that $a \leq b$, the stated assumption allows us to denote the knight as a $(1,2)$-leaper and, consequently, there are infinitely many fairy chess leapers (see Table~1).

\begin{table}[H]
    \centering
    \begin{tabular}{|c|c|c|c|c||c|c}
\hline
 \backslashbox{$\boldsymbol b$}{$\boldsymbol a $} & $\boldsymbol 0$ & $\boldsymbol 1$ & $\boldsymbol 2$ & $\boldsymbol 3$ & $\boldsymbol 4$ & \dots \\
\hline 
$\boldsymbol 1$ & Wazir (W) & Ferz (F)  & --  & --  & --  &  \\
\hline
$\boldsymbol 2$ & Dabbaba (D) & Knight (N) & Alfil (A) & --  & --  &  \\
\hline 
$\boldsymbol 3$ & Threeleaper (H) & Camel (C) & Zebra (Z) & Tripper (G)  & --  &  \\
\hline
\hline
$\boldsymbol 4$ & Fourleaper & Giraffe & Stag & Antelope & Commuter &  \\
\hline
\vdots &  &  &  &  &  &   \\

\end{tabular}
    \caption{List of $(a,b)$-leapers in fairy chess such that $a \geq 0$, $b \geq 1$, and $a \leq b$.}
    \label{Tab1}
\end{table}

By focusing our research on the $k$-dimensional $2 \times 2 \times \cdots \times 2$ chessboards, we have recently investigated some fairy chess leapers, providing sufficient conditions for the existence of closed tours (e.g., Theorem~3.3 of \cite{8} states that a closed Euclidean $(2,3)$-leaper tour is possible as long as $k \geq 15$).

Let $C(2,k) \coloneqq \{(x_1,x_2,\dots,x_k) \mid x_i \in \{0,1\},\, 1 \leq i \leq k \}$ denote the vertex set of the $k$-dimensional hypercube. This set consists of $2^k$ 
distinct vertices, each corresponding to a binary $k$-tuple.

More generally, let $C(n,k) \coloneqq \{0,1,\ldots,n-1\}^k$.

\begin{definition} \label{hamiltonian}
A Hamiltonian cycle in $C(2,k)$ (i.e., in the grid $\{0,1\}^k$) is a closed Euclidean tour in $C(2,k)$ that includes the $2^k$-th move to return to the starting vertex.
\end{definition}

In the present paper, we extend the previous result by providing a necessary and sufficient condition via two theorems. The first focuses on the existence of Hamiltonian cycles in the $k$-dimensional grid $C(2,k) \coloneqq \{0,1\}^k$, for any associated Euclidean distance. The second is more specific, as it provides a necessary and sufficient condition for the existence of closed Euclidean $(a,b)$-leaper tours in $2 \times 2 \times \cdots \times 2$ chessboards, with associated distance $\sqrt{a^2+b^2}$.
In fact, for any given $(a,b)$-leaper with nonnegative $a$, positive $b$, and $a \leq b$ (to ensure uniqueness), we consider the Euclidean distance $\sqrt{a^2+b^2}$ associated with the only leaper whose move takes one vertex of the $k$-cube to another at exact Euclidean distance $\sqrt{a^2+b^2}$. Consequently, every fairy chess leaper is well defined (e.g., the condition $a \leq b$ implies that the threeleaper corresponds to $(0,3)$, not $(3,0)$).

Consider two vertices of $C(2,k)$, $v_i \coloneqq (x_1,x_2,\dots,x_k)$ and $v_j \coloneqq (y_1,y_2,\dots,y_k)$, where $x_1,\dots,x_k \in \{0,1\}$ and $y_1,\dots,y_k \in \{0,1\}$.

The Euclidean distance between $v_i$ and $v_j$, $\| v_i-v_j \|: C(2,k) \rightarrow \mathbb{R}$, is
\begin{equation} \label{norm}
\| v_i-v_j \| \coloneqq \sqrt{(x_1-y_1)^2+(x_2-y_2)^2+\cdots+(x_k-y_k)^2}.
\end{equation}

Given $C(2,k)$, any $(a,b)$-leaper moving from $v_i$ to $v_j$ (and vice versa) implies that $\| v_i-v_j \|$ equals $\sqrt{a^2+b^2}$. For example, for $k=5$, consider the vertices $v_1=(0,0,0,0,0)$ and $v_2=(1,1,1,1,1)$. Then the $(1,2)$-leaper, called the Knight, is allowed to move between $v_1$ and $v_2$, since
\begin{equation} 
\| v_2-v_1 \| = \sqrt{(1-0)^2+(1-0)^2+(1-0)^2+(1-0)^2+(1-0)^2}=\sqrt{1^2+2^2}=\sqrt{5}.
\end{equation}

Thus, a closed Euclidean tour in $C(2,k)$ is a sequence of $2^k$ distinct vertices of $C(2,k)$ such that the distance between each pair of consecutive elements is always the same, and where also the distance between the ending vertex and the starting one is equal to this fixed value.

If the distance $\sqrt{a^2+b^2}$ is given, for the construction of a closed Euclidean tour in $C(2,k)$, we need to find a sequence, let us denote it by $P^{a^2+b^2}_\textnormal{C}(2,k) \coloneqq v_1 \rightarrow v_2 \rightarrow \dots \rightarrow v_{2^k}$, such that $\|v_i-v_{i+1}\|=\sqrt{a^2+b^2}$ is true for each integer $i \in \{1,2^k-1\}$ and which satisfies the ``closed tour'' condition $\| v_{2^k}-v_1\|=\sqrt{a^2+b^2}$.

The notation $P^{a^2+b^2}_\textnormal{C}(2,k)$ is coherent with Definition~1.2 of \cite{8}, the only difference is that now the apex (i.e., $a^2+b^2$) describes the squared Euclidean distance associated with the Hamiltonian cycle (and no longer the selected fairy chess leaper) while the subscript $\textnormal{C}$ still indicates a closed Euclidean tour. Clearly, $(2,k)$ is the argument that arises from the $k$-dimensional hypercubes we are considering in this paper (i.e., the grids of the form $C(2,k)$). 

For each grid $C(2,k)$, let us note that the distance between any two of its vertices is the square root of the number of coordinates switched between them (e.g., let $v_1, v_2 \in C(2,k)$ be such that $v_1 \equiv (0,1,1,0,1)$ and $v_2 \equiv (0,1,0,1,0)$, it follows that their Euclidean distance is $\|v_1-v_2\|=\sqrt{0^2+0^2+1^2+1^2+1^2}=\sqrt{3}$).

Then, since we are only considering grids of the form $\{0,1\}^k$, we can specify any movement rule by counting how many times a \emph{change of coordinates} (i.e., $0\mapsto1$ or $1\mapsto0$, as appropriate) occurs when moving from one vertex of $C(2,k)$ to the next.
\begin{definition} \label{change}
Let $k \in \mathbb{N}$ and let $h$ be a nonnegative integer such that $h \leq k$ and assume that both $v_i$ and $v_j$ belong to $\{0,1\}^k$. Then, we define the transformation $v_i \rightarrow v_j$ performed by changing exactly $h$ coordinates of $v_i$ as ``change $h$'' so that $\| v_i - v_j \| = \sqrt{h}$.
\end{definition}

The change-$1$ case originates the \emph{reflected binary code} (RBC) \cite{9}, which has been indirectly used in \cite{8}, where we proved the existence of wazir's tours in $C(2,k)$ for each positive integer $k$.

For example, given the polygonal chain $P^1_\textnormal{C}(2,2) \coloneqq (0,0) \rightarrow (1,0) \rightarrow (1,1) \rightarrow (0,1)$ that joins all the vertices of $C(2,2)$, it is possible to find a sequence at change $1$ in $C(2,3)$ by proceeding as follows:
\begin{enumerate}
    \item We add a new coordinate on the right-hand side of each term to construct the two polygonal chains $S_1 \coloneqq (0,0,0) \rightarrow (1,0,0) \rightarrow (1,1,0) \rightarrow (0,1,0)$ and $S_2 \coloneqq (0,0,1) \rightarrow (1,0,1) \rightarrow (1,1,1) \rightarrow (0,1,1)$ belonging to $\mathbb{R}^3$.
    \item We invert $S_2$ to get the polygonal chain $\hat{S_2} \coloneqq (0,1,1) \rightarrow (1,1,1) \rightarrow (1,0,1) \rightarrow (0,0,1) \in \{0,1\}^3$.
    \item We join, one after the other, $S_1$ and $\hat{S_2}$ to obtain $P^1_\textnormal{C}(2,3) \coloneqq (0,0,0) \rightarrow (1,0,0) \rightarrow (1,1,0) \rightarrow (0,1,0) \rightarrow (0,1,1) \rightarrow (1,1,1) \rightarrow (1,0,1) \rightarrow (0,0,1)$, which is a closed Euclidean tour for $C(2,3)$ as the distance between its starting and final vertex is $1$.
\end{enumerate}

Therefore, we have constructively shown the existence of change $1$ Euclidean tours for all the grids of the form $C(2,k)$ (i.e., it is possible to join all the vertices of each $k$-dimensional hypercube with a polygonal chain whose line segments have a unit Euclidean length and such that we can add one more edge of equal length to finally get a closed polygonal chain \cite{6, 8}).

\section{Existence of Hamiltonian cycles} \label{sec:2}
Several authors have investigated Hamiltonian properties of hypercubes under various constraints (see, for instance, \cite{2}), but their approach remains purely combinatorial and does not involve Euclidean distances between vertices.

\sloppy The nonexistence of Hamiltonian cycles depends on the concept of parity.

In Subsection~4.1 of~\cite{4} the definition of even and odd vertices is introduced in the following way: assuming $m~\in~\mathbb{N}_0$, for any given $k$-tuple $(x_1, x_2, \dots, x_k)$ of nonnegative integers smaller than $n$, $(x_1, x_2, \dots, x_k)$ is an \emph{even} vertex of $\{0, 1, \ldots, n-1\}^k$ if and only if
\begin{equation}
    \sum_{j=1}^k x_j = 2 m,
\end{equation} 
otherwise we say that $(x_1, x_2, \dots, x_k)$ is \emph{odd}.

The following theorem adopts the distinction between odd and even vertices to show that, for each pair of integers $(n,k)$ with $n,k > 1$, there are infinitely many $(a,b)$-leapers that cannot perform any Hamiltonian cycle in $\{0, 1, \ldots, n-1\}^k$.
Here, as usual, we assume that $a \in \mathbb{N}_0$, $b \in \mathbb{N}$, and $a \leq b$.
\begin{theorem} \label{parity}
\textnormal{\bf{(From Theorem~2.1 in \cite{8}, Section~2).}} Let $n,k \in \mathbb{N} \setminus \{1\}$ so that the $k$-dimensional grid $C(n,k)$ is given. Then, consider any $(a,b)$-leaper such that $a+b$ is even. If such a leaper is originally placed on an even starting vertex, it can only visit (some of) the $\ceil[\Big]{\frac{n^{k}}{2}}$ even vertices of $C(n,k)$, whereas if the $(a,b)$-leaper starts from an odd vertex, it can only visit (some of) the $\floor[\Big]{\frac{n^{k}}{2}}$ odd vertices.
\end{theorem}

In the notation of this paper, for our purposes, Theorem~\ref{parity} can be improved as Lemma~\ref{parity2}.
\begin{lemma} \label{parity2}
For each even positive integer $h$ and each integer $k > h$, there are no Hamiltonian cycles in $C(2,k)$ whose associated Euclidean distance is $\sqrt{h}$.
\end{lemma}
\begin{proof}
In accordance with the notation adopted in~\cite{8}, let $d_1, d_2, \ldots, d_h$ denote the coordinates that are equal to $1$ in a given vertex of $C(2,k)$ (the remaining $k-h$ coordinates are equal to $0$).
Since the Euclidean distance $\sqrt{h}$ implies a change $h$, the statement of Lemma~\ref{parity2} trivially follows from the proof of Theorem~2.1 in \cite{8}. In fact, if the starting vertex is even, as we change $h$ of its $k$ coordinates, we get an even vertex. \linebreak On the other hand, when we start from an odd vertex, the change $h$ results in another odd vertex.
\end{proof}

For any given integer $k>1$, the condition $h<k$ immediately follows from the fact that the maximum Euclidean distance between every two elements of $\{0,1\}^k$ equals $\sqrt{k}$.

Now, let us call the \emph{opposite corner} of a vertex $(x_1, \dots, x_k)$ in $C(2,k)$ the unique vertex at Euclidean distance $\sqrt{k}$ from it, namely the one identified by the $k$-tuple $(1-x_1, \dots, 1-x_k)$. Observe that each vertex of $C(2,k)$ has exactly one opposite corner.

Consequently, we can state the nonexistence lemma below.
\begin{lemma} \label{noexistence}
For each pair of positive integers $(k, h)$ with $h \geq k > 1$, there are no Hamiltonian cycles in $C(2,k)$ whose associated Euclidean distance is $\sqrt{h}$.
\end{lemma}

\begin{proof}

Since each element of $C(2,k)$ is uniquely defined by a $k$-tuple of $0$'s and/or $1$'s, the Euclidean distance between any two vertices of $C(2,k)$ belongs to the set $\mathrm{D} = \{\sqrt{k-k}, \sqrt{k-(k-1)}, \ldots, \sqrt{k-(k-k)}\}$, whose maximum is $\sqrt{k}$.

By hypothesis, $\sqrt{h} \geq \sqrt{k} > 1$. If $\sqrt{h} > \sqrt{k}$, then the associated Euclidean distance $\sqrt{h}$ does not belong to the set $\mathrm{D}$. Otherwise, $\sqrt{h}=\sqrt{k}$, and we can only move from any given vertex of $C(2,k)$, say $(x_1, \dots, x_k)$, to its opposite corner (i.e., the one characterized by the same coordinates, swapped according to the transformation $(x_1, \dots, x_k) \mapsto (1-x_1, \dots, 1-x_k)$), since the Euclidean distance between these two vertices is exactly $\sqrt{k}$.

Thus, only the mentioned pair of vertices can be covered, leaving out the remaining $2^k-2$ vertices of $C(2,k)$.
\end{proof}

We are finally ready to state the existence Theorem~\ref{existence}.
\begin{theorem} \label{existence}
For each integer $k > 1$ and each odd integer $h \in [1, k-1]$, there exists a Hamiltonian cycle in $C(2,k)$ whose associated Euclidean distance is $\sqrt{h}$.
\end{theorem}
\begin{proof}
Since we aim to prove the existence of a Hamiltonian cycle in $C(2,k)$, by Lemmas~\ref{parity2} and \ref{noexistence}, the only configurations that we need to consider are those in which $h$ is an odd positive integer smaller than $k$. Using the described wazir tour algorithm, at change $1$, we find the polygonal chain $P^1_\textnormal{C}(2,k) \coloneqq v_0 \rightarrow v_1 \rightarrow \dots \rightarrow v_{2^k-1}$ that visits exactly once all the vertices of $C(2,k)$ and which is characterized by a unit distance between its endpoints. From Definition~\ref{change}, it follows that the Euclidean distance associated with the path mentioned above is $\sqrt{1}$ (i.e., $\sqrt{h}$ where $h \coloneqq 1$) and so, by changing only one coordinate at a time, we induce an alternation between even and odd vertices as we move from any vertex of $C(2,k)$ to the next one. Without loss of generality, in the sequence of vertices described by $P^1_\textnormal{C}(2,k)$, let $v_j$ be even if and only if $j$ is even (and vice versa if $j$ is odd).
As we want to modify the distance between every pair of consecutive vertices from the sequence above, we take all the odd vertices of $P^1_\textnormal{C}(2,k)$ and switch all their coordinates (i.e., $0 \mapsto 1$ and $1 \mapsto 0$).

This little trick performs a change of $k-1$ coordinates at each step, which yields a path in $\{0,1\}^k$ whose associated Euclidean distance is $\sqrt{k-1}$. When $k$ is odd, $k-1$ is even; hence a change of $k-1$ coordinates preserves parity. Starting from an even (respectively, odd) vertex, one can only reach even (respectively, odd) vertices, which precludes a Hamiltonian cycle. On the other hand, an even value of $k$ implies that $k-1$ is odd, so the operation of changing the coordinates of the odd vertices is equivalent to permuting the odd vertices over the entire sequence, preserving the fundamental property required to obtain a valid Hamiltonian cycle in $C(2,k)$. Then, let $h \coloneqq k-1$. We have a Hamiltonian cycle whose associated Euclidean distance is $\sqrt{h}$.

Now, let $k'$ be a positive integer smaller than $k$ and assume that the aforementioned algorithm has returned a Hamiltonian cycle in $C(2,k')$ with the associated distance $\sqrt{k'-1}$ (i.e., $\sqrt{h}$, where $h \coloneqq k'-1$). The extended algorithm (doubling-and-switching construction) will allow us to find a new Hamiltonian cycle in $C(2,k'+1)$ which is also characterized by an associated (Euclidean) distance of $\sqrt{h}$:
\begin{enumerate}
    \item We duplicate the Hamiltonian path in $C(2,k')$ (i.e., we take a Hamiltonian cycle in $\{0,1\}^{k'}$ and then we remove one of its edges) by adding a new coordinate at the right-hand side of each term of the original sequence of vertices so that we create a new pair of polygonal chains in $\mathbb{R}^{k'+1}$, $S_1$ (whose ($k'+1$)-th coordinate is $0$) and $S_2$ (whose ($k'+1$)-th coordinate is $1$). Although both $S_1$ and $S_2$ belong to $C(2,k'+1)$ and their union contains all the vertices of $\{0,1\}^{k'+1}$, they cannot yet be joined into a single (valid) Hamiltonian cycle. In fact, the distance between the first vertices of $S_1$ and $S_2$ (and likewise between their last vertices) is equal to $1$.
    \item Focusing on $S_2$, we carry out a change of the leftmost $h-1$ coordinates of each vertex and, accordingly, we call this new sequence of $2^{k'}$ vertices $S_3$. We now observe that the distance between the first vertices of $S_1$ and $S_3$ (and likewise between their last vertices) is $\sqrt{k'-1}$. Then, let $h \coloneqq k'-1$ and observe that in $S_3$ the associated Euclidean distance between consecutive vertices remains the same as in $S_2$, which is always equal to $\sqrt{h}$ (since $h=k'-1$).
    \item We reverse the order of the vertices of $S_3$ (i.e., let the last vertex of $S_3$ become the first one, and so forth). Let us call $S_4$ this new polygonal chain, and then observe that the transformation we have performed on $S_3$ does not affect the distance so that we can still get a Hamiltonian path in $C(2,k')$ which is associated with the distance $\sqrt{k'-1}$. 
    \item We connect the paths $S_1$ and $S_4$ by inserting a segment of Euclidean length $\sqrt{k'-1}=\sqrt{h}$ between the last vertex of $S_1$ and the first vertex of $S_4$. Indeed, the first vertex of $S_4$ is obtained from the last vertex of $S_1$ by changing exactly $h$ coordinates (namely, the $h-1$ leftmost coordinates and the rightmost one), and the same holds for the last vertex of $S_4$ and the first vertex of $S_1$. The resulting path can then be closed by joining its endpoints, yielding a Hamiltonian cycle in $C(2,k'+1)$ with associated Euclidean distance $\sqrt{h}$.
\end{enumerate}

By changing the value of $h$ for all odd integers smaller than $k$ and repeating the extended algorithm $k-k'$ times, we get a Hamiltonian cycle in $C(2,k)$ for all odd positive integers $h \in \{1, 2,\ldots, k-1\}$.
\end{proof}

We can use the algorithms mentioned above to find a Hamiltonian cycle in $C(2,5)$ with associated Euclidean distance $\sqrt{3}$.
In detail, we generate a wazir's tour in $C(2,4)$ by applying the recursive algorithm several times, as shown below.

Let us start from the closed wazir's tour
\begin{align*} 
P^1_\textnormal{C}(2,4)  & \coloneqq (0,0,0,0) \rightarrow (1,0,0,0) \rightarrow (1,1,0,0) \rightarrow (0,1,0,0) \\
& \rightarrow (0,1,1,0) \rightarrow (1,1,1,0) \rightarrow (1,0,1,0) \rightarrow (0,0,1,0) \\
& \rightarrow (0,0,1,1) \rightarrow (1,0,1,1) \rightarrow (1,1,1,1) \rightarrow (0,1,1,1) \\ 
& \rightarrow (0,1,0,1) \rightarrow  (1,1,0,1) \rightarrow (1,0,0,1) \rightarrow (0,0,0,1).
\end{align*}

\sloppy Then, we change all the coordinates of the odd vertices of $P^1_\textnormal{C}(2,4)$ to obtain the new Hamiltonian path $P^3_\textnormal{C}(2,4)$ with associated Euclidean distance equal to $\sqrt{4-1}$:
\begin{align*}
P^3_\textnormal{C}(2,4) & \coloneqq
(0,0,0,0) \rightarrow (0,1,1,1) \rightarrow (1,1,0,0) \rightarrow (1,0,1,1) \\ 
& \rightarrow (0,1,1,0) \rightarrow (0,0,0,1) \rightarrow (1,0,1,0) \rightarrow (1,1,0,1) \\
& \rightarrow (0,0,1,1) \rightarrow (0,1,0,0) \rightarrow (1,1,1,1) \rightarrow (1,0,0,0) \\
& \rightarrow (0,1,0,1) \rightarrow (0,0,1,0) \rightarrow (1,0,0,1) \rightarrow (1,1,1,0).
\end{align*}

Applying the doubling-and-switching construction to $P^3_\textnormal{C}(2,4)$, we get two new distinct Hamiltonian paths in $C(2,5)$ with associated Euclidean distance $\sqrt{3}$ (i.e., $S_1$ and $S_2$, see below):
\begin{align*}
S_1  & \coloneqq (0,0,0,0,0) \rightarrow (0,1,1,1,0)\rightarrow (1,1,0,0,0) \rightarrow (1,0,1,1,0) \\
&  \rightarrow (0,1,1,0,0) \rightarrow  (0,0,0,1,0) \rightarrow (1,0,1,0,0) \rightarrow (1,1,0,1,0) \\
& \rightarrow (0,0,1,1,0) \rightarrow (0,1,0,0,0) \rightarrow (1,1,1,1,0)\rightarrow  (1,0,0,0,0) \\
&  \rightarrow (0,1,0,1,0) \rightarrow (0,0,1,0,0) \rightarrow (1,0,0,1,0) \rightarrow (1,1,1,0,0);  \\[2mm]
S_2  & \coloneqq (0,0,0,0,1) \rightarrow (0,1,1,1,1)\rightarrow (1,1,0,0,1) \rightarrow (1,0,1,1,1) \\
& \rightarrow (0,1,1,0,1)\rightarrow  (0,0,0,1,1) \rightarrow (1,0,1,0,1) \rightarrow (1,1,0,1,1)\\
& \rightarrow (0,0,1,1,1) \rightarrow (0,1,0,0,1) \rightarrow (1,1,1,1,1)\rightarrow  (1,0,0,0,1) \\
& \rightarrow (0,1,0,1,1) \rightarrow (0,0,1,0,1)\rightarrow (1,0,0,1,1) \rightarrow (1,1,1,0,1).
\end{align*}

We change the $3-1$ leftmost coordinates of each vertex of $S_2$ (following the rule $0 \mapsto 1$ and $1 \mapsto 0$, as usual) to generate the new sequence
\begin{align*}
S_3 & \coloneqq (1,1,0,0,1) \rightarrow (1,0,1,1,1) \rightarrow (0,0,0,0,1)\rightarrow (0,1,1,1,1)  \\
& \rightarrow (1,0,1,0,1) \rightarrow  (1,1,0,1,1)\rightarrow (0,1,1,0,1) \rightarrow (0,0,0,1,1) \\
& \rightarrow (1,1,1,1,1)\rightarrow (1,0,0,0,1) \rightarrow (0,0,1,1,1) \rightarrow  (0,1,0,0,1) \\
& \rightarrow (1,0,0,1,1) \rightarrow (1,1,1,0,1) \rightarrow (0,1,0,1,1)\rightarrow (0,0,1,0,1).
\end{align*}

Now, let us reverse the order of the vertices of $S_3$; we start from its last vertex (i.e., $(0,0,1,0,1)$) and go backwards toward the first one (i.e., $(1,1,0,0,1)$) so that we create
\begin{align*}
S_4 & \coloneqq (0,0,1,0,1) \rightarrow (0,1,0,1,1)\rightarrow (1,1,1,0,1) \rightarrow (1,0,0,1,1) \\
& \rightarrow (0,1,0,0,1)\rightarrow  (0,0,1,1,1) \rightarrow (1,0,0,0,1) \rightarrow (1,1,1,1,1) \\
& \rightarrow (0,0,0,1,1) \rightarrow (0,1,1,0,1) \rightarrow (1,1,0,1,1)\rightarrow  (1,0,1,0,1) \\
&  \rightarrow (0,1,1,1,1) \rightarrow (0,0,0,0,1) \rightarrow (1,0,1,1,1) \rightarrow (1,1,0,0,1).
\end{align*}

Finally, we join the polygonal chains $S_1$ and $S_4$ through the oriented line segment $(1,1,1,0,0) \rightarrow (0,0,1,0,1)$ so that we obtain a Hamiltonian path in $C(2,5)$ with associated (Euclidean) distance $\sqrt{3}$:
\begin{align*}
P^3_\textnormal{C}(2,5) & \coloneqq (0,0,0,0,0) \rightarrow (0,1,1,1,0)\rightarrow (1,1,0,0,0) \rightarrow (1,0,1,1,0) \\
& \rightarrow (0,1,1,0,0)\rightarrow  (0,0,0,1,0) \rightarrow (1,0,1,0,0) \rightarrow (1,1,0,1,0) \\
& \rightarrow (0,0,1,1,0) \rightarrow (0,1,0,0,0) \rightarrow (1,1,1,1,0)\rightarrow  (1,0,0,0,0) \\
& \rightarrow (0,1,0,1,0) \rightarrow (0,0,1,0,0)\rightarrow (1,0,0,1,0) \rightarrow (1,1,1,0,0) \\
& \rightarrow (0,0,1,0,1)\rightarrow  (0,1,0,1,1) \rightarrow (1,1,1,0,1) \rightarrow (1,0,0,1,1) \\
& \rightarrow (0,1,0,0,1)\rightarrow (0,0,1,1,1) \rightarrow (1,0,0,0,1) \rightarrow  (1,1,1,1,1) \\
& \rightarrow (0,0,0,1,1) \rightarrow (0,1,1,0,1)\rightarrow (1,1,0,1,1) \rightarrow (1,0,1,0,1) \\
& \rightarrow (0,1,1,1,1) \rightarrow  (0,0,0,0,1) \rightarrow (1,0,1,1,1) \rightarrow (1,1,0,0,1).
\end{align*}

Lastly, $P^3_\textnormal{C}(2,5) \cup \{(1,1,0,0,1) \rightarrow (0,0,0,0,0)\}$ is a valid Hamiltonian cycle in $\{0,1\}^5$ whose associated Euclidean distance is $\sqrt{3}$.

From Theorem~\ref{existence} we can easily determine the existence of Hamiltonian cycles for every fairy chess leaper.
More specifically, since our candidates are all the $(a,b)$-leapers satisfying $a \geq 0$ and $b \geq a > 0$ (see Table~\ref{Tab1}), we can invoke Theorem~\ref{existence} to state a necessary and sufficient condition on the existence of Hamiltonian cycles in $C(2,k)$, where the associated Euclidean distance is mandatorily equal to the square root of $a^2+b^2$.
\begin{corollary} \label{NSC}
Let $a \geq 0$, $b \geq a > 0$, and $k > 1$ be positive integers. A necessary and sufficient condition for the existence of closed Euclidean $(a,b)$-leaper tours in $C(2,k)$ is provided by $a+b \equiv 1 \pmod 2$ and $k > a^2+b^2$.
\end{corollary}
\begin{proof}
By Theorem~\ref{parity}, $a+b$ cannot be even.
Thus, $a+b$ is odd, and then the proof is trivial since $a+b \equiv 1 \pmod 2$ implies $a^2+b^2 \equiv 1 \pmod 2$.

Accordingly, let us assign $h \coloneqq a^2+b^2$ so that we can invoke Theorem~\ref{existence} to show that a Hamiltonian cycle in $C(2,k)$ exists for any $(a,b)$-leaper as above if and only if $1 \leq h \leq k-1$.

Therefore, $k$ is greater than $h$ and the statement of Corollary~\ref{NSC} follows.
\end{proof}

\section{Conclusion} \label{Concl}
For each integer $k > 1$ and each odd integer $h \in [1, k-1]$, Theorem~\ref{existence} gives a sufficient condition for the existence of Hamiltonian cycles in $C(2,k) \coloneqq \{0,1\}^k$ with associated Euclidean distance $\sqrt{h}$. Concretely, its proof describes the algorithm that generates a valid Hamiltonian cycle, as specified above, starting from a given wazir's tour \cite{6}.

Furthermore, assuming that $k$ is greater than $1$, Corollary~\ref{NSC} proves the existence of closed $(a,b)$-leaper tours in $C(2,k)$ if and only if $a+b$ is odd and $k > a^2+b^2$.

This corollary fills the gap left by Theorems~3.2 and 3.3 of~\cite{8} since it proves the existence of a threeleaper's tour in $C(2,10)$ and a zebra's tour in $C(2,14)$.

\section*{Acknowledgments}
The authors sincerely thank Nereus Duruemezuo for his constant support.

\makeatletter
\renewcommand{\@biblabel}[1]{[#1]\hfill}
\makeatother

\end{document}